\documentclass[reqno]{amsart}
\usepackage{xcolor} 
\usepackage{amsthm,amsmath,amssymb}
\usepackage{stmaryrd,mathrsfs}
\usepackage[OT2,T1]{fontenc}
\usepackage{esint}

\allowdisplaybreaks

\newcommand{\eps}[0]{\varepsilon}

\def\cyr{\fontencoding{OT2}\fontfamily{wncyr}\selectfont}
\DeclareTextFontCommand{\textcyr}{\cyr}

\swapnumbers \numberwithin{equation}{section}

\theoremstyle{plain}
\newtheorem{theorem}[equation]{Theorem}
\newtheorem{proposition}[equation]{Proposition}
\newtheorem{corollary}[equation]{Corollary}
\newtheorem{lemma}[equation]{Lemma}

\theoremstyle{definition}
\newtheorem{definition}[equation]{Definition}

\theoremstyle{remark}
\newtheorem{remark}[equation]{Remark}

\makeatletter
\@namedef{subjclassname@2010}{%
  \textup{2010} Mathematics Subject Classification}
\makeatother

\usepackage{mathtools}
\mathtoolsset{showonlyrefs,showmanualtags} 

\usepackage{hyperref} 
\hypersetup{
    colorlinks=true,       
    linkcolor=blue,          
    citecolor=magenta,        
    filecolor=magenta,      
    urlcolor=cyan           
}

\begin{document}

\title{The $A_p$--$A_\infty$ inequality for general Calder\'on--Zygmund operators}

\author[T.~P.\ Hyt\"onen]{Tuomas P.\ Hyt\"onen}
\address{Department of Mathematics and Statistics, P.O.B.~68 (Gustaf H\"all\-str\"omin katu~2b), FI-00014 University of Helsinki, Finland}
\email{tuomas.hytonen@helsinki.fi}

\author[M.~T.\ Lacey]{Michael T.\ Lacey}
\address{School of Mathematic, Georgia Institute of Technology, Atlanta GA 30332, U.S.A.}
\email{lacey@math.gatech.edu}

\date{\today}



\maketitle

\begin{abstract}
Let $T$ be an arbitrary $L^2$ bounded Calder\'on--Zygmund operator, and $T_{\natural}$ its maximal truncated version. 
Then it satisfies the following bound for all $p\in(1,\infty)$ and all $w\in A_p$:
\begin{equation*}
  \|T_{\natural}f\|_{L^p(w)}
  \leq C_{T,p}[w]_{A_p}^{1/p}\big([w]_{A_\infty}^{1/p'}+[w^{1-p'}]_{A_\infty}^{1/p}\big)\|f\|_{L^p(w)}.
\end{equation*}
\end{abstract}

\section{Introduction}

For a weight $w$ with a non-negative density a.e., denote
\begin{equation*}
\begin{split}
  A_p(w,Q): &=\frac{w(Q)}{|Q|}\Big(\frac{w^{1-p'}(Q)}{|Q|}\Big)^{p-1},\qquad p\in(1,\infty),\\
  A_\infty(w,Q) &:=\frac{1}{w(Q)}\int_Q M(w1_Q), \\
  [w]_{A_p} &:=\sup_Q A_p(w,Q),\qquad p\in(1,\infty].
\end{split}
\end{equation*}

We prove the following theorem:

\begin{theorem}\label{thm:ApAinfty}
Let $T$ be an arbitrary $L^2$ bounded Calder\'on--Zygmund operator, and $T_{\natural}$ its maximal truncated version. Then it satisfies the following bound for all $p\in(1,\infty)$ and all $w\in A_p$:
\begin{equation*}
  \|T_{\natural}f\|_{L^p(w)}
  \leq C_{T,p}[w]_{A_p}^{1/p}\big([w]_{A_\infty}^{1/p'}+[w^{1-p'}]_{A_\infty}^{1/p}\big)\|f\|_{L^p(w)}.
\end{equation*}
\end{theorem}

We refer the reader to the cited papers for a definition of an `$L^2$ bounded Calder\'on--Zygmund operator, and its maximal truncated version.'  
This result improves and generalizes the following weighted inequalities previously known for arbitrary Calder\'on--Zygmund operators:
\begin{itemize}
  \item The $A_p$ conjecture for untruncated operators, obtained in \cite{Hytonen:A2}:
\begin{equation*}
  \|Tf\|_{L^p(w)}\leq C_{T,p}\big([w]_{A_p}+[w]_{A_p}^{1/(p-1)}\big)\|f\|_{L^p(w)},\qquad p\in(1,\infty).
\end{equation*}  
  \item Its extension to maximal truncations, and to a preliminary form of the mixed $A_p$--$A_\infty$ bounds in \cite{HLMORSU}:
\begin{equation*}
  \|T_{\natural}f\|_{L^p(w)}
  \leq C_{T,p}\big([w]_{A_p}^{1/p}[w]_{A_\infty}^{1/p'}+[w]_{A_p}^{1/(p-1)}\big)\|f\|_{L^p(w)},\qquad p\in(1,\infty).  
\end{equation*}
  \item The case $p=2$ of the Theorem for untruncated operators, from \cite{HytPer}:
\begin{equation*}
    \|T f\|_{L^2(w)}
  \leq C_{T}[w]_{A_2}^{1/2}\big([w]_{A_\infty}^{1/2}+[w^{-1}]_{A_\infty}^{1/2}\big)\|f\|_{L^2(w)}.
\end{equation*}
\end{itemize}

The bound of Theorem~\ref{thm:ApAinfty} was previously proven in \cite{Lacey:ApAinfty} for a restricted class of distinguished Calder\'on--Zygmund operators, including the Hilbert, Riesz and Beurling transforms; the Theorem as stated was also conjectured there.

The present proof consists of elaborating on the themes of the previous papers. As in \cite{Lacey:ApAinfty}, Lerner's local oscillation formula from \cite{Lerner:formula} is used to reduce the estimation of a dyadic shift operator to related positive operators of the form
\begin{equation*}
  f\mapsto\sum_{Q\in\mathscr{Q}}1_Q\fint_{Q^{(i)}}f,
\end{equation*}
where $\mathscr{Q}$ is an appropriately sparse collection of dyadic cubes and $Q^{(i)}$ is the $i$th ancestor of $Q$. In \cite{Lacey:ApAinfty}, these operators were further dominated by simpler positive operators previously studied in \cite{LSU:positive}. The novelty of the present contribution consists of studying the new operators as above directly on their own right. In fact, we observe that these operators are not so new after all; they are just special cases of dyadic shifts of complexity $i$, but with a positive kernel. So we are in a position to apply the results for general shifts from \cite{HLMORSU}, where their weighted $L^p$ estimates were reduced to certain testing conditions. Now, these testing conditions simplify radically in the positive kernel case, and this simplification gives us the additional advantage to verify the sharper mixed bounds as stated.

The method of proof, which goes via the dual-weight formulation, provides the following two-weight generalization, where
\begin{equation*}
  A_p(w,\sigma;Q):=\frac{w(Q)}{|Q|}\Big(\frac{\sigma(Q)}{|Q|}\Big)^{p-1},\qquad [w,\sigma]_{A_p}:=\sup_Q A_p(w,\sigma;Q).
\end{equation*}

\begin{theorem}
Let $T$ be an arbitrary $L^2$ bounded Calder\'on--Zygmund operator, and $T_{\natural}$ its maximal truncated version. Then it satisfies the following bound for all $w,\sigma\in A_\infty$:
\begin{equation*}
  \|T_{\natural}(f\sigma)\|_{L^p(w)}
  \leq C_{T,p}[w,\sigma]_{A_p}^{1/p}\big([w]_{A_\infty}^{1/p'}+[\sigma]_{A_\infty}^{1/p}\big)\|f\|_{L^p(\sigma)}.
\end{equation*}
\end{theorem}

\begin{remark}
For the same class of distinguished operators as mentioned above, Lerner \cite{Lerner:ApAr} has recently established a different type of a mixed $A_p$--$A_r$ bound, using the product-type weight constant
\begin{equation*}
  [w]_{A_p^{\alpha}A_r^{\beta}}:=\sup_Q A_p(w,Q)^{\alpha}A_r(w,Q)^{\beta},
\end{equation*}
which is smaller that $[w]_{A_p}^{\alpha}[w]_{A_r}^{\beta}$ in that only one supremum, rather than independent ones over the two factors above, is involved. He proves that
\begin{equation*}
  \|T_{\natural}f\|_{L^p(w)}\leq C_{T,p,r}[w]_{A_p^{1/(p-1)}A_r^{1-1/(p-1)}}\|f\|_{L^p(w)},\qquad 2\leq p<r<\infty,
\end{equation*}
and shows by examples that this is incomparable with our bound. A ultimate conjecture generalizing both results would be the following:
\begin{equation*}
  \|T_{\natural}f\|_{L^p(w)}\leq C_{T,p}\big([w]_{A_p^{1/p}A_\infty^{1/p'}}+[w^{1-p'}]_{A_{p'}^{1/p'}A_\infty^{1/p}}\big)\|f\|_{L^p(w)},\qquad 1<p<\infty.
\end{equation*}

The present proof does not seem to allow product-type weight constants as here; on the technical level, the difficulty is in a pigeonholing construction, which consists of ``freezing'' the local $A_p$ constant
\begin{equation*}
  A_p(w,Q)=\frac{w(Q)}{|Q|}\Big(\frac{w^{1-p'}(Q)}{|Q|}\Big)^{p-1}.
\end{equation*}
Once we also freeze the ratio $w(Q)/|Q|$, it follows that also the ratio $w^{1-p'}(Q)/|Q|$, too, is automatically under control. This would not be the case, had we frozen the product $A_p(w,Q)^{\alpha}A_r(w,Q)^{\beta}$ instead, for then there would still be two independent measure ratios $w^{1-p'}(Q)/|Q|$ and $w^{1-r'}(Q)/|Q|$ (or $\int_Q M(w1_Q)\big/w(Q)$ if $r=\infty$) around.
\end{remark}

\section{Reduction to positive dyadic shifts}

Let $\mathscr{D}$ be the collection of dyadic cubes in $\mathbb{R}^d$. For $Q\in\mathscr{D}$, we write $\ell(Q)=|Q|^{1/d}$ for its side-length, and $Q^{(1)}$ for its parent: the unique dyadic cube such that $\ell(Q^{(1)})=2\ell(Q)$ and $Q^{(1)}\supset Q$. The dyadic ancestors are defined inductively: $Q^{(i)}:=Q^{(i-1)}$.

The following definitions are taken from Section 2 of \cite{HLMORSU}. 
\begin{definition}
Let $Q\in\mathscr{D}$ be a dyadic cube. A \emph{generalized Haar function} associated to $Q$ is a linear combination of the indicator functions of its dyadic children:
\begin{equation*}
h_{Q}=\sum_{\substack{Q'\in\mathscr{D}\\ (Q')^{(1)}=Q}}c_{Q^{\prime }}1_{Q^{\prime }}.
\end{equation*}
We say $h_{Q}$ is a \emph{Haar function} if in addition $\int h_{Q}=0$, that is, a Haar function is
orthogonal to constants on its support. 
\end{definition}

\begin{definition}\label{d.haarShift} For integers $(m,n) \in \mathbb Z _+ ^2 $, we say that  a linear operator $ \mathbb S $ is a \emph{(generalized) Haar shift operator of   complexity type $ (m,n)$} if 
\begin{equation}\label{e.mn}
  \mathbb S  f (x) = \sum_{Q \in \mathcal D}\mathbb{S}_Q f(x)
  = \sum_{Q \in \mathcal D}\quad\sideset {} { ^ {(m,n)}} \sum_{\substack{Q',R'\in \mathcal D\\ Q',R'\subset Q }} 
\frac { \langle f, h ^{Q'} _{R'} \rangle} {\lvert  Q\rvert } k^{R'} _{Q'}(x)
\end{equation}
where
\begin{itemize}
  \item in the second sum, the superscript $ ^{(m,n)}$ on the sum means that in addition we require $ \ell (Q') = 2 ^{-m} \ell (Q)$ and $ \ell (R')= 2 ^{-n} \ell (Q)$,
  \item the function $  h ^{Q'} _{R'}$ is a (generalized) Haar function on $ R'$, and $  k ^{R'} _{Q'}$ is one on $ Q'$, with the normalization that 
\begin{equation} \label{e.normal}
\lVert  h ^{Q'} _{R'}\rVert_{\infty }\leq 1,\qquad \lVert  k ^{R'} _{Q'}\rVert_{\infty } \leq 1 \,,
\end{equation}
  \item  on the unweighted $L^2$ space, the operator satisfies $\|\mathbb{S}f\|_{L^2}\leq\|f\|_{L^2}$.
\end{itemize}
We say that the \emph{complexity} of $ \mathbb S $ is $  \kappa := \max (m,n,1)$.  
\end{definition}

A generalized Haar shift thus has the form%
\begin{equation*}
\mathbb{S}f\left( x\right) =\sum_{Q\in \mathcal{ D}}\frac{1}{\left\vert
Q\right\vert }\int_{Q}s_{Q}\left( x,y\right) f\left( y\right) dy=\int_{%
\mathbb{R}^{n}}K_{\mathbb S}\left(x,y\right) f\left( y\right) dy,
\end{equation*}%
where $s_{Q}$, the kernel of the component $\mathbb{S}_Q$, is supported on $Q\times Q$ and $\left\Vert s_{Q}\right\Vert
_{\infty }\leq 1$. It is easy to check that
\begin{equation*} 
  |K_{\mathbb{S}}(x,y)|\lesssim\frac{1}{|x-y|^d}.
\end{equation*}
The role of \emph{positive} dyadic operators is essential in this note.   

Thanks to the representation of an $L^2$ bounded Calder\'on--Zygmund operator established in \cite{Hytonen:A2}, 
our main Theorem follows from this one.  (For more details on this reduction, see section 2 of \cite{HLMORSU}.) 

\begin{proposition}\label{p.shift}  Let $ \mathbb S $ be a generalized Haar shift operator of complexity $ \kappa $, and $ 1<p< \infty $. 
It holds that 
\begin{equation*}
  \| \mathbb S_{\natural}f\|_{L^p(w)}
  \leq C_{p} \kappa ^2 [w]_{A_p}^{1/p}\big([w]_{A_\infty}^{1/p'}+[w^{1-p'}]_{A_\infty}^{1/p}\big)\|f\|_{L^p(w)}.
\end{equation*}
\end{proposition}

We note that any polynomial dependence on complexity $ \kappa $ is sufficient to conclude our paper.

The application of  Lerner's formula \cite{Lerner:formula} to Haar shift operators 
is well-known, see \cite{CMP-ERAMS}, \cite{CMP}, \cite{Lerner:Ap}, \cite{Lacey:ApAinfty}.  
But all previous applications have given bounds that are exponential in $ \kappa $.  
Lerner's formula  gives the pointwise bound
\begin{equation*}
  |f-m_f(Q_0) |
  \lesssim M^{\#}_{1/4;Q_0}f+\sum_{k,j}\omega_{2^{-d-2}}(f;(Q^{k}_j)^{(1)}) 1_{Q^k_j},
\end{equation*}
where the various quantities are defined as follows:
\begin{gather}
 \label{e.w}
\omega _{\lambda } (\phi  ; Q) := 
\inf _{c\in \mathbb R }  \bigl( (\phi - c )1_{Q}  \bigr) ^{\ast} (\lambda \lvert  Q\rvert ) \,,
\\\label{e.localSharp}
M ^{\sharp} _{\lambda; Q } \phi  (x) 
:= \sup _{Q'\subset Q} 1_{Q'} \omega_{\lambda } (\phi , Q')\,, 
\end{gather}
and $ m_f(Q_0)$ is a median of $ f$ restricted to $ Q_0$, a possibly non-unique real number such that 
\begin{equation*}
\max\bigl\{ \lvert \{ x \in Q_0 \;:\;  f(x)>m_f(Q_0) \}\rvert,\ 
\lvert \{ x \in Q_0 \;:\;  f(x)<m_f(Q_0) \}\rvert \bigr\} \le \tfrac 12 \lvert  Q_0\rvert.  
\end{equation*}
In addition, $ \phi ^{\ast} $ denotes the non-increasing rearrangement, so that if $ \phi $ is supported on $ Q$, 
$ \phi ^{\ast} (\lambda \lvert  Q\rvert )$ is the $ \lambda ^{\textup{th}}$  percentile of $ \phi $. 
Importantly, defining $\Omega_k:=\bigcup_j Q^k_j$  as a disjoint union, it holds that  $\Omega_{k+1}\subset\Omega_k$ and $|Q^k_j\cap\Omega_{k+1}|\leq\tfrac12|Q^k_j|$.

Let $f$ be bounded with bounded support, and $Q_N\to\mathbb{R}^d$. Since $f\in L^2$, we have $\mathbb{S}_{\natural}f\in L^2$, hence $m_{Q_N}(\mathbb{S}_{\natural}f)\to 0$, and finally $1_{Q_N}(\mathbb{S}_{\natural}f-m_{\mathbb{S}_{\natural}f}(Q_N))\to \mathbb{S}_{\natural}f$ pointwise. By Fatou's lemma and Lerner's formula (applied to $Q_N$ in place of $Q_0$), then
\begin{equation*}
\begin{split}
  \int_{\mathbb{R}^d}(\mathbb{S}_{\natural}f)^p w
  &\leq\liminf_{N\to\infty}\int_{Q_N}|\mathbb{S}_{\natural}f-m_{\mathbb{S}_{\natural}f}(Q_N)|^p w \\
  &\lesssim\liminf_{N\to\infty}\Big(\int_{Q_N}(M^{\#}_{1/4;Q_N}f)^p w \\
  &\qquad  +\int_{Q_N}\Big[\sum_{k,j}1_{Q^k_j}\omega_{2^{-d-2}}(\mathbb{S}_{\natural}f;(Q^k_j)^{(1)})\Big]^p w\Big).
\end{split}
\end{equation*}

\begin{lemma}
If $\mathbb{S}$ has complexity $\kappa$, then
\begin{equation*}
  \omega_{\lambda}(\mathbb{S}_{\natural}f;Q)
  :=\inf_c (1_Q(\mathbb{S}_{\natural}f-c))^*(\lambda|Q|)
  \lesssim \kappa \fint_{Q^{(0)}}|f| +  
  \sum_{i=1}^{\kappa}\fint_{Q^{(i)}}|f|.
\end{equation*}
\end{lemma}

\begin{proof}
We have
\begin{equation*}
\begin{split}
  1_Q\mathbb{S}_{\natural}f
  &:=1_Q\sup_{\epsilon<\upsilon}\Big|\sum_{\epsilon\leq\ell(Q')\leq\upsilon}\mathbb{S}_{Q'} f\Big| \\
  &\phantom{:}\leq1_Q\sup_{\epsilon<\upsilon}\Big|\sum_{\substack{\epsilon\leq\ell(Q')\leq\upsilon \\ Q'\subseteq Q}}\mathbb{S}_{Q'} f\Big|
     +1_Q\sum_{Q':Q\subsetneq Q'\subseteq Q^{(\kappa)}}|\mathbb{S}_{Q'} f| \\
  &\qquad   +1_Q\sup_{\epsilon<\upsilon}\Big|\sum_{\substack{\epsilon\leq\ell(Q')\leq\upsilon \\ Q'\supsetneq Q^{(\kappa)}}}\mathbb{S}_{Q'} f\Big|.
\end{split}
\end{equation*}
The last term is a constant (say, $c_Q$) times $1_Q$. Hence
\begin{equation*}
  1_Q|\mathbb{S}_{\natural}f-c_Q|
  \leq 1_Q\mathbb{S}_{\natural}(1_Q f)+1_Q\sum_{i=1}^{\kappa}\fint_{Q^{(i)}}|f|.
\end{equation*}
It follows that
\begin{equation*}
\begin{split}
  \inf_c(1_Q(\mathbb{S}_{\natural}f-c))^*(\lambda|Q|)
  &\leq(1_Q(\mathbb{S}_{\natural}f-c_Q))^*(\lambda|Q|) \\
  &\leq(\mathbb{S}_{\natural}(1_Q f))^*(\lambda|Q|)+\sum_{i=1}^{\kappa}\fint_{Q^{(i)}}|f|,
\end{split}
\end{equation*}
where finally
\begin{equation*}
  (\mathbb{S}_{\natural}(1_Q f))^*(\lambda|Q|)
  \leq\frac{1 }{\lambda|Q|}\|\mathbb{S}_{\natural}(1_Q f)\|_{L^{1,\infty}}
  \lesssim\frac{\kappa }{|Q|}\|1_Q f\|_1=\kappa\fint_Q|f|.
\end{equation*}
\end{proof}

The lemma implies in particular that
\begin{align*}
  M^{\#}_{1/4;Q_N}(\mathbb{S}_{\natural}f)
  \lesssim \sup_{Q\subseteq Q_N}\kappa 1_Q   \fint_{Q}|f|+ 
  \sum_{i=1}^{\kappa}\fint_{Q^{(i)}}|f|
  \leq 2\kappa Mf,
\end{align*}
and hence
\begin{equation*}
  \int_{Q_N}M^{\#}_{1/4;Q_N}(\mathbb{S}_{\natural}f)^p w
  \lesssim\kappa^p\int (Mf)^p w
  \lesssim\kappa^p[w]_{A_p}[w^{1-p'}]_{A_\infty}\int|f|^p w
\end{equation*}
by the mixed bound for the maximal operator from \cite{HytPer}. This bound is of the correct form.

Let us write $\mathscr{Q}:=\{Q^k_j:k,j\in\mathbb{N}\}$. It remains to consider
\begin{equation*}
  \sum_{Q\in\mathscr{Q}} 1_Q\omega_{2^{-d-2}}(\mathbb{S}_{\natural}f;(Q)^{(1)})
  \lesssim\sum_{Q\in\mathscr{Q}}1_Q\Big(\kappa\fint_{Q^{(1)}}|f|+\sum_{i=2}^{\kappa+1}\fint_{Q^{(i)}}|f|\Big),
\end{equation*}
where further, for each $i=1,\ldots,\kappa+1$,
\begin{equation*}
  \sum_{Q\in\mathscr{Q}}1_Q\fint_{Q^{(i)}}|f|
  =\sum_{Q\in\mathscr{D}}\Big(\sum_{\substack{Q'\in\mathscr{Q}\\ (Q')^{(i)}=Q}}1_{Q'}\Big)\fint_Q|f| 
  =:\sum_{Q\in\mathscr{D}}\chi_Q^{(i)}\fint_Q|f| =:\mathbb{S}^{(i)}|f|.
\end{equation*}

Here $\mathbb{S}^{(i)}$ is a \emph{positive} generalized dyadic shift of complexity $i$: each $\mathbb{S}^{(i)}_Q f=\chi_Q ^{(i)}\fint_Q f$ is a positive operator with kernel $|Q|^{-1}\chi^{(i)}_Q\otimes 1_Q$. (We have not yet checked the $L^2$ boundedness, though, which is not automatic in the noncancellative case.) 

It hence suffices to show that
\begin{equation}\label{e.toShow}
  \int (\mathbb{S}^{(i)}f)^p w\lesssim i^p\cdot [w]_{A_p}([w]_{A_\infty}^{p-1}+[w^{1-p'}]_{A_\infty})\int f^p w,
\end{equation}
where $f\geq 0$. This supplies our quadratic in complexity bound in Proposition~\ref{p.shift} above. 
This bound to be proven might look just as bad as what we started with, but the positivity gives us a significant advantage.

\section{A two-weight inequality for positive dyadic shifts}

We recall the two-weight inequality for maximal truncations of general shifts:

\begin{theorem}[\cite{HLMORSU}]  \label{t.HLMORSU}
Let $\mathbb{S}$ be a generalized dyadic shift of complexity $\kappa$. Then
\begin{equation*}
  \|\mathbb{S}_{\natural}(f\sigma)\|_{L^p(w)}
  \lesssim\{\kappa\mathfrak{M}_p+\mathfrak{N}_p+\mathfrak{T}_p\}\|f\|_{L^p(\sigma)},
\end{equation*}
where $\mathfrak{M}_p,\mathfrak{N}_p,\mathfrak{T}_p$ are the constants from the following three estimates, where
\begin{equation*}
  \mathbb{L}f(x):=e^{i\vartheta(x)}\sum_{\eps(x)\leq\ell(Q)\leq\upsilon(x)}\mathbb{S}_Q f(x)
\end{equation*}
ranges over all possible linearizations of $\mathbb{S}_{\natural}$:
\begin{equation*}
\begin{split}
  \|M( &f\sigma) \|_{L^p(w)} \leq\mathfrak{M}_p\|f\|_{L^p(\sigma)}, \qquad
      \|1_Q\mathbb{L}^*(1_Q gw)\|_{L^{p'}(\sigma)}\leq\mathfrak{T}_p w(Q)^{1/p'}\|g\|_{\infty} \\
  &\Big(\int_{Q_0}\sup_{Q\subset Q_0}1_Q\Big[\frac{1}{w(Q)}\int_Q|\mathbb{L}^*(1_Q g w)|\sigma\Big]^p w\Big)^{1/p}
    \leq\mathfrak{N}_p\sigma(Q_0)^{1/p}\|g\|_{\infty}.
\end{split}
\end{equation*}
\end{theorem}

We now prove:

\begin{corollary}
Let $\mathbb{S}$ be a positive dyadic shift of complexity $\kappa$. Then
\begin{equation*}
  \|\mathbb{S}(f\sigma)\|_{L^p(w)}
  \lesssim\{\kappa\mathfrak{M}_p+\mathfrak{S}_p+\mathfrak{S}_p^*\}\|f\|_{L^p(\sigma)},
\end{equation*}
where $\mathfrak{M}_p$ is as above, and $\mathfrak{S}_p,\mathfrak{S}_p^*$ are the constants from the following two estimates:
\begin{equation*}
      \|1_Q\mathbb{S}(1_Q \sigma)\|_{L^{p}(w)}\leq\mathfrak{S}_p \sigma(Q)^{1/p},\qquad
      \|1_Q\mathbb{S}^*(1_Q w)\|_{L^{p'}(\sigma)}\leq\mathfrak{S}_p^* w(Q)^{1/p'}.
\end{equation*}
\end{corollary}

In fact, we check that $\mathfrak{N}_p\lesssim\mathfrak{S}_p$ and $\mathfrak{T}_p\leq\mathfrak{S}_p^*$ for positive shifts. The key point is that when $\mathbb{S}$ is a positive shift, any linearization $\mathbb{L}$ satisfies $|\mathbb{L}^*g|\leq\mathbb{S}^*|g|$ pointwise. From this, it is immediate that
\begin{equation*}
\begin{split}
  \|1_Q\mathbb{L}^*(1_Q gw)\|_{L^{p'}(\sigma)}
  &\leq\|1_Q\mathbb{S}^*(1_Q |g|w)\|_{L^{p'}(\sigma)} \\
  &\leq\|1_Q\mathbb{S}^*(1_Q w)\|_{L^{p'}(\sigma)}\|g\|_{\infty}
  \leq\mathfrak{S}_p^*w(Q)^{1/p'}\|g\|_{\infty},
\end{split}
\end{equation*}
and thus $\mathfrak{T}_p\leq\mathfrak{S}_p^*$.

\begin{lemma}
If $\mathbb{S}$ is a positive shift, then
\begin{equation*}
  \mathfrak{N}_p\lesssim\mathfrak{S}_p.
\end{equation*}
\end{lemma}

\begin{proof}
Let $\|g\|_{\infty}\leq 1$. We have
\begin{equation*}
  \int_Q|\mathbb{L}^*(1_Q g w)|\sigma
  \leq\int_Q \mathbb{S}^*(1_Q w)\sigma
  =\int_Q\mathbb{S}(1_Q\sigma)w
  \leq\int_Q\mathbb{S}(1_{Q_0}\sigma)w,
\end{equation*}
and hence
\begin{equation*}
  1_Q\Big[\frac{1}{w(Q)}\int_Q|\mathbb{L}^*(1_Q g w)|\sigma\Big]^p
  \leq 1_Q\Big[\frac{1}{w(Q)}\int_Q \mathbb{S}(1_{Q_0} \sigma)w \Big]^p
  \leq 1_Q M_w(1_{Q_0}\mathbb{S}(1_{Q_0}\sigma))^p,
\end{equation*}
where $M_w$ is the dyadic maximal operator with respect to the measure $w$. Finally,
\begin{equation*}
  \Big(\int_{Q_0} M_w(1_{Q_0}\mathbb{S}(1_{Q_0}\sigma))^p w\Big)^{1/p}
  \lesssim\Big(\int_{Q_0}\mathbb{S}(1_{Q_0}\sigma)^p w\Big)^{1/p}
  \leq\mathfrak{S}_p\sigma(Q_0)^{1/p}.\qedhere
\end{equation*}
\end{proof}

\begin{remark}
The fact that we deduce a positive-operator result from one for singular operators is somewhat unusual. One could obviously give a direct proof of the Corollary, but this would be only somewhat simpler than the proof of the Theorem~\ref{t.HLMORSU}, and certainly much harder than the above deduction based on the Theorem.
\end{remark}

\section{The unweighted boundedness of the particular shifts $\mathbb{S}^{(i)}$}

We now return to the question of unweighted $L^2$ boundedness of the particular positive shifts
\begin{equation*}
  \mathbb{S}^{(i)}f=\sum_{Q\in\mathscr{Q}}\chi_Q^{(i)}\fint_Q f,\qquad\chi_Q^{(i)}:=\sum_{\substack{Q'\in\mathscr{Q}\\ (Q')^{(i)}=Q}}1_{Q'}
\end{equation*}
arising from the application of Lerner's formula to general shifts.    
This proposition will show that $ c i ^{-1} \mathbb S ^{(i)}$ 
fulfills the requirements of the definition of a generalized Haar shift.

\begin{proposition}
When $\mathscr{Q}=\{Q^k_j:k,j\in\mathbb{N}\}$ are the cubes from Lerner's formula, we have for $p\in(1,\infty)$,
\begin{equation*}
  \|\mathbb{S}^{(i)}f\|_{L^p}\lesssim i\|f\|_{L^p}.
\end{equation*}
\end{proposition}
  
\begin{proof}
This could be done in a variety of ways.
One possibility is to apply the two-weight result in the case that $w\equiv\sigma\equiv 1$. Then $\mathfrak{M}_p\lesssim 1$ by the usual maximal inequality. Let us drop the superscript $(i)$ for simplicity in the subsequent analysis.
As for $\mathfrak{S}_p$, we have
\begin{equation*}
  1_Q \mathbb{S}(1_Q)
  =\sum_{R\subseteq Q}\chi_R +\sum_{R\supsetneq Q}1_Q\chi_R\frac{|Q|}{|R|}
  \leq \sum_{R\subseteq Q}\chi_R+1.
\end{equation*}
Here,
\begin{equation*}
  \sum_{R\subseteq Q}\chi_R=\sum_{k=0}^{\infty} 1_{\bigcup\mathscr{Q}_k(Q)},
\end{equation*}
where $\mathscr{Q}_0(Q)$ is the collection of the maximal $Q'\in\mathscr{Q}$ with $(Q')^{(i)}\subseteq Q$, and inductively $\mathscr{Q}_{k}(Q)$ is the collection of the maximal $Q'\in\mathscr{Q}$ strictly contained in some $Q''\in\mathscr{Q}_{k-1}(Q)$. From the properties of Lerner's cubes, we have
\begin{equation*}
  \Big|\bigcup\mathscr{Q}_{k}(Q)\Big|\leq 2^{-1}\Big|\bigcup\mathscr{Q}_{k-1}(Q)\Big|\leq\ldots\leq 2^{-k}|\bigcup\mathscr{Q}_{0}(Q)\Big|\leq 2^{-k}|Q|,
\end{equation*}
and hence
\begin{equation*}
    \Big\|\sum_{R\subseteq Q}\chi_R\Big\|_p  \leq \sum_{k=0}^{\infty} \Big|\bigcup\mathscr{Q}_k(Q)\Big|^{1/p}\lesssim|Q|^{1/p}.
\end{equation*}

We turn to $\mathfrak{S}_p^*$. First,
\begin{equation*}
  \mathbb{S}^*g=\sum_{Q\in\mathscr{D}} 1_Q\fint_Q\chi_Q f,
\end{equation*}
and hence
\begin{equation*}
  1_Q\mathbb{S}^*(1_Q)
  =\sum_{R\subseteq Q} 1_R\fint_R\chi_R+\sum_{R\supsetneq Q} 1_Q\frac{1}{|R|}\int_Q\chi_R
  \leq \sum_{R\subseteq Q} 1_R\fint_R\chi_R+1.
\end{equation*}
Estimating the first term by duality with $f\in L^p$, we have
\begin{equation*}
\begin{split}
  \int\Big(\sum_{R\subseteq Q} 1_R\fint_R\chi_R\Big)f
  &=\int\sum_{R\subseteq Q}\chi_R\fint_R f \\
  &\leq\int\Big(\sum_{R\subseteq Q}\chi_R \Big)Mf 
  \lesssim\Big\|\sum_{R\subseteq Q}\chi_R \Big\|_{p'}\|Mf\|_{p}
  \lesssim|Q|^{1/p'}\|f\|_{p}
\end{split}
\end{equation*}
by the previous part of the proof and the maximal theorem.

\end{proof}

\section{The testing constants $\mathfrak{S}_p$ and $\mathfrak{S}_p^*$ in the two-weight case}

In this section we provide an estimate of the testing constants
\begin{equation*}
  \mathfrak{S}_p:=\sup_Q\frac{\|1_Q\mathbb{S}(1_Q\sigma)\|_{L^p(w)}}{\sigma(Q)^{1/p}},\qquad
  \mathfrak{S}_p^* :=\sup_Q\frac{\|1_Q\mathbb{S}^*(1_Q w)\|_{L^{p'}(\sigma)}}{w(Q)^{1/p'}}
\end{equation*}
for an arbitrary (not necessarily positive) dyadic shift $\mathbb{S}$ of complexity $\kappa$. But recall that these are only known to dominate the $L^p$ norm bounds of $\mathbb{S}$ in the special case of positive shifts (or, by different methods, general shifts but only for $p=2$). For the application to or main results, we are ultimately interested in the special case that $\mathbb{S}=\mathbb{S}^{(i)}$.

Let us fix a cube $Q_0$ and the (by now usual) decomposition of its subcubes: Let $\mathscr{K}$ be one of the $\kappa+1$ subcollections of
\begin{equation}\label{e.sepScales}
  \{Q\subseteq Q_0:\log_2\ell(Q)\equiv\lambda\mod\kappa+1\}, 
\end{equation}
and
\begin{equation*}
  \mathscr{K}^a:=\{Q\in\mathscr{K}: 2^a\leq\Big(\frac{w(Q)}{|Q|}\Big)^{1/p}\Big(\frac{\sigma(Q)}{|Q|}\Big)^{1/p'}<2^{a+1}\Big\},
\end{equation*}
where
\begin{equation*}
  2^a\leq [w,\sigma]_{A_p}^{1/p}:=\sup_Q \Big(\frac{w(Q)}{|Q|}\Big)^{1/p}\Big(\frac{\sigma(Q)}{|Q|}\Big)^{1/p'}.
\end{equation*}
Let $\mathscr{P}^a=\bigcup_{n=1}^{\infty}\mathscr{P}^a_n\subseteq\mathscr{K}^a$ be the principal cubes such that $\mathscr{P}^a_0$ consists of all maximal cubes in $\mathscr{K}^a$, and $\mathscr{P}^a_n$ consists of all maximal $P'$ contained in some $P\in\mathscr{P}^a_{k-1}$ with the estimate
\begin{equation*}
  \frac{\sigma(P')}{|P'|}>2\frac{\sigma(P)}{|P|}.
\end{equation*}
For $Q\in\mathscr{K}^a$, let $\Pi(Q)$ be the minimal principal cube containing it, and
\begin{equation*}
  \mathscr{K}^a(P):=\{Q\in\mathscr{K}^a:\Pi(Q)=P\}.
\end{equation*}

For general shifts (and even their maximal truncations), the following distributional estimate is known; see \cite{HLMORSU} Section 10 
in the $A_p$ setting. We remark that the proof is substantially  easier in the non-negative case, the only case we actually need for our main results.

\begin{lemma}[\cite{HLMORSU}]
For any dyadic shift of arbitrary complexity, we have
\begin{equation*}
  w\Big(|\mathbb{S}_{\mathscr{K}^a(P)}(\sigma)|>t\frac{\sigma(P)}{|P|}\Big)\lesssim e^{-ct}w(P),\qquad P\in\mathscr{P}^a,
\end{equation*}
where $c$ is a dimensional constant.
\end{lemma}

The following estimate completes the verification of the testing conditions. 


\begin{proposition}
Let $\mathbb{S}$ be a dyadic shift of complexity $\kappa$. Then
\begin{equation*}
  \|1_Q\mathbb{S}(1_Q\sigma)\|_{L^p(w)}\lesssim(1+\kappa)\big([w,\sigma]_{A_p}[\sigma]_{A_\infty}\sigma(Q)\big)^{1/p}
\end{equation*}
\end{proposition}

\begin{proof}
We write
\begin{equation*}
  1_Q\mathbb{S}(1_Q\sigma)
  =\sum_{R\subseteq Q}\mathbb{S}_R(\sigma)+1_Q\sum_{R\supsetneq Q}\mathbb{S}_R(1_Q\sigma),
\end{equation*}
where the second term satisfies
\begin{equation*}
  \Big|1_Q\sum_{R\supsetneq Q}\mathbb{S}_R(1_Q\sigma)\Big|
  \leq 1_Q\sum_{R\supsetneq Q}\frac{\sigma(Q)}{|R|}\leq 1_Q\frac{\sigma(Q)}{|Q|}
\end{equation*}
and
\begin{equation*}
  \Big\|1_Q\frac{\sigma(Q)}{|Q|}\Big\|_p =w(Q)^{1/p}\frac{\sigma(Q)}{|Q|}\leq[w,\sigma]_{A_p}^{1/p}\sigma(Q)^{1/p}.
\end{equation*}
The first term we write as
\begin{equation*}
  \sum_{R\subseteq Q}\mathbb{S}_R(\sigma)
  =\sum_{\lambda=0}^{\kappa}\sum_{a:2^a\leq[w,\sigma]_{A_p}^{1/p}}\sum_{P\in\mathscr{P}^a}\mathbb{S}_{\mathscr{K}^a(P)}(\sigma),
\end{equation*}
where the dependence of the quantities on the parameter $\lambda$ from \eqref{e.sepScales} is suppressed.  For $P\in\mathscr{P}^a$, define the set 
\begin{equation*}
  P_j^a:=\Big\{|\mathbb{S}_{\mathscr{K}^a(P)}(\sigma)|\in \frac{\sigma(P)}{|P|}(j,j+1]\Big\}\subseteq P.
\end{equation*}
By the distributional estimate, we have $w(P_j^a)\lesssim e^{-cj}w(P)$. Then
\begin{equation*}
  \Big\|\sum_{R\subseteq Q}\mathbb{S}_R(\sigma)\Big\|_{L^p(w)}
  \leq\sum_{\lambda=0}^{\kappa}\sum_{a}\sum_{j=0}^{\infty}(j+1)
     \Big\|\sum_{P\in\mathscr{P}^a}1_{P^a_j}\cdot\frac{\sigma(P)}{|P|}\Big\|_{L^p(w)},
\end{equation*}
and
\begin{equation*}
\begin{split}
  \Big\|&\sum_{P\in\mathscr{P}^a} 1_{P^a_j}\cdot\frac{\sigma(P)}{|P|}\Big\|_{L^p(w)}
  =\Big(\int\Big[\sum_{P\in\mathscr{P}^a}1_{P^a_j}(x)\cdot\frac{\sigma(P)}{|P|}\Big]^p w(dx)\Big)^{1/p} \\
  &\overset{(*)}{\lesssim} \Big(\int \sum_{P\in\mathscr{P}^a}1_{P^a_j}(x)\cdot\Big[\frac{\sigma(P)}{|P|}\Big]^p w(dx)\Big)^{1/p} 
    =\Big(\int \sum_{P\in\mathscr{P}^a}w(P^a_j)\cdot\Big[\frac{\sigma(P)}{|P|}\Big]^p\Big)^{1/p} \\
  &\lesssim \Big(\int \sum_{P\in\mathscr{P}^a}e^{-cj}w(P)\cdot\Big[\frac{\sigma(P)}{|P|}\Big]^p\Big)^{1/p} 
  \leq 2^a\Big(e^{-cj}\sum_{P\in\mathscr{P}^a}\sigma(P)\Big)^{1/p} \\
  &\overset{(**)}{\lesssim} 2^a\Big(e^{-cj}[\sigma]_{A_\infty}\sigma(Q)\Big)^{1/p},
\end{split}
\end{equation*}
where we in $(*)$ the fact that the numbers $\frac{\sigma(P)}{|P|}$, for $P\owns x$ with a fixed $x$, form a super-exponential sequence (so that their $\ell^1$ and $\ell^p$ norms a comparable) and in $(**)$ an estimate for the principal cubes in terms of $[\sigma]_{A_\infty}$ from \cite{HytPer}.

Now we can simply sum up
\begin{equation*}
\begin{split}
  \Big\|\sum_{R\subseteq Q}\mathbb{S}_R(\sigma)\Big\|_{L^p(w)}
  &\leq\sum_{\lambda=0}^{\kappa}\sum_{a:2^a\leq[w,\sigma]_{A_p}^{1/p}}\sum_{j=0}^{\infty}(j+1)
    \Big\|\sum_{P\in\mathscr{P}_a} 1_{P^a_j}\cdot\frac{\sigma(P)}{|P|}\Big\|_{L^p(w)} \\
  &\lesssim\sum_{\lambda=0}^{\kappa}\Big(\sum_{a:2^a\leq[w,\sigma]_{A_p}^{1/p}}2^a\Big)\Big(\sum_{j=0}^{\infty}(j+1)2^a e^{-cj/p}\Big)[\sigma]_{A_\infty}^{1/p}\sigma(Q)^{1/p} \\
  &\lesssim(\kappa+1)[w,\sigma]_{A_p}^{1/p}[\sigma]_{A_\infty}^{1/p}\sigma(Q)^{1/p}.\qedhere
\end{split}
\end{equation*}
\end{proof}

\bibliography{weighted}
\bibliographystyle{plain}

\end{document}